\documentclass[a4paper,11pt]{article}
\usepackage{amssymb,amsmath,amsthm,amsfonts}
\usepackage{mathrsfs}
\usepackage{multirow}
\usepackage{graphicx}
\usepackage{authblk}
\usepackage{indentfirst}
\usepackage{multicol}
\usepackage{url}
\usepackage{fancyhdr}
\usepackage[numbers]{natbib}
\usepackage[all]{xy}
\usepackage{mathtools}
\usepackage[english]{babel}
\usepackage{colortbl}
\usepackage{caption}
\usepackage{subcaption}
\usepackage{tikz}
\usepackage{float}
\usepackage{enumitem}
\usepackage{listings}
\usepackage{hyperref} % keep only this for links
\usepackage{makecell}

\newtheorem{theorem}{Theorem}[section]
\newtheorem{proposition}[theorem]{Proposition}
\newtheorem{lemma}[theorem]{Lemma}
\newtheorem{corollary}[theorem]{Corollary}

\newtheorem{definition}{Definition}[section]
\newtheorem{example}{Example}[section]

\usepackage{xr}
\makeatletter
    \newcommand*{\addFileDependency}[1]{
    \typeout{(#1)}
    \@addtofilelist{#1}
    \IfFileExists{#1}{}{\typeout{No file #1.}}
    }
\makeatother

\title{Computing Khovanov homology of tangles}

\author[1]{Li Shen$^{\ast}$}
\author[2,3]{Jian Liu$^{\ast}$}
\author[3,4,5]{Guo-Wei Wei\thanks{Corresponding author: weig@msu.edu}}
\affil[1]{NSF-Simons National Institute for Theory and Mathematics in Biology, Chicago IL.}
\affil[2]{Mathematical Science Research Center, Chongqing University of Technology, Chongqing 400054, China}
\affil[3]{Department of Mathematics, Michigan State University, MI 48824, USA}
\affil[4]{Department of Electrical and Computer Engineering, Michigan State University, MI 48824, USA}
\affil[5]{Department of Biochemistry and Molecular Biology, Michigan State University, MI 48824, USA}

\makeatletter
    \renewcommand*{\@fnsymbol}[1]{\ensuremath{\ifcase#1\or \dagger\or *\or *\or
   \mathsection\or \else\@ctrerr\fi}}
\makeatother
\date{}

\begin{document}
    % \linenumbers
    \maketitle
    \footnotetext[2]{Shen and Liu contributed equally to this work.
    The work of Liu was done during his two-year stay at Michigan State University.}

    \paragraph{Abstract}
    The computation of Khovanov homology for tangles has significant potential applications, yet explicit computational studies remain limited. In this work, we present a method for computing the Khovanov homology of tangles via an arc reduction approach, and we derive the Poincar\'e polynomial for simple tangles. Furthermore, we compute the Poincar\'e polynomials of tangles with at most three crossings.



\tableofcontents % insert contents

\section{Introduction}

Khovanov homology, introduced by Khovanov \cite{khovanov2000categorification}, is a homological refinement of the Jones polynomial that provides a fundamental invariant for the study of knots and low-dimensional topology. For links, Khovanov homology has been extensively investigated from various perspectives, including computational approaches, structural properties, and applications to topological invariants.

Tangles, which can be viewed as local fragments of links, are subject to more relaxed constraints than links, making them particularly suitable for studying both local and global properties of links and their associated algebraic structures. This also highlights the greater potential of tangles for various applications.
The Khovanov homology of tangles plays a crucial role in studying tangle invariants and understanding their local properties \cite{bar2002khovanov,bar2005khovanov,khovanov2002functor}.
Moreover, Morrison, Walker, Rozansky, and collaborators \cite{clark2009fixing,manolescu2023skein, morrison2022invariants} studied tangle Khovanov homology, including constructions based on skein modules, functoriality under cobordisms, and extensions to higher-dimensional topological invariants.
In \cite{liu2024persistent}, the authors studied the persistent Khovanov homology of tangles, highlighting the practical significance of computing the Khovanov homology of tangles for various applications. The work in \cite{shen2025algorithm} proposes algorithms for computing the Khovanov homology of tangles.

In the original construction, D. Bar-Natan defined a cochain complex for tangles over the additive category, which is obtained from the category of cobordisms through a sophisticated construction.
In \cite{liu2024persistent}, a functor from the above additive category to the abelian category of $\mathbb{k}$-modules is given. This functor induces a cochain complex in the abelian category, providing an effective method for computing the Khovanov homology of tangles.

In this work, we further investigate the computation and properties of the Khovanov homology of tangles. Specifically, for a tangle $T$, suppose there exists an arc that intersects the rest of the tangle at a single crossing, and let the tangle formed by the remaining part be denoted by $T_1$. We then have the following theorem (see Theorem \ref{theorem:arc_reduction}):
\begin{theorem}
If the added crossing is right-handed, then
\[
   H^{k,q}(T) \;\cong\; H^{k,q}(T_1) \oplus H^{k-1,q-1}(T_1),
\]
If the added crossing is left-handed, then
\[
   H^{k,q}(T) \;\cong\; H^{k+1,q+3}(T_1) \oplus H^{k,q+2}(T_1).
\]
Here, $k$ denotes the homological grading and $q$ denotes the quantum grading.
\end{theorem}

In addition, we define the notion of a simple tangle. Roughly speaking, a simple tangle is one in which no subset of arcs forms a closed region. We then obtain the following result (see Theorem \ref{theorem:poincare}).
\begin{theorem}
Let $T$ be a simple tangle consisting of $N$ arcs, endowed with an orientation. Then the Poincar\'e polynomial of $T$ is given by
\[
  \mathcal{P}_{T}(x,y) = y^{-N+n_{+}+n_{-}} (1+xy)^{n_{+}} \,(x^{-1}y^{-3}+y^{-2})^{n_{-}}.
\]
Here, $n_{+}$ and $n_{-}$ denote the numbers of right-handed and left-handed crossings in the tangle, respectively.
\end{theorem}

Finally, we classify tangles with at most three crossings and compute the Poincar\'e polynomial corresponding to each type.

The paper is organized as follows. In the next section, we present the main results. In Section \ref{section:classification}, we classify tangles with at most three crossings and compute their corresponding Khovanov homology.

%\section{preliminaries}
%½éÉÜK-module, crossing, state cube,±¾ÎÄÖÐËùÓеÄtangleÖ¸µÄÊÇtangle diagram

\section{Main results}

In this section, we present our main results, including the behavior of the Khovanov homology of tangles under a certain arc reduction, as well as the Khovanov homology of simple tangles. For convenience, we take the ground ring to be a coefficient field $\mathbb{K}$.

\subsection{Arc reduction with one crossing}

\begin{theorem}\label{theorem:arc_reduction}
Let $T$ be a tangle, and let $T'$ be the tangle obtained from $T$ by adding an arc with a single crossing, either below or above $T$.
\begin{itemize}
\item If the added crossing is right-handed, then
\[
   H^{k,q}(T') \;\cong\; H^{k,q}(T) \oplus H^{k-1,q-1}(T),
\]
\item If the added crossing is left-handed, then
\[
   H^{k,q}(T') \;\cong\; H^{k+1,q+3}(T) \oplus H^{k,q+2}(T).
\]
\end{itemize}
Here, $k$ denotes the homological grading and $q$ denotes the quantum grading.
\end{theorem}

\begin{proof}
Since $T'$ has one more crossing than $T$, we place this additional crossing at the first coordinate of the state cube $\{0,1\}^{n+1}$, where $n$ denotes the number of crossings of $T$. Observe that $T$ together with the arc produces exactly one crossing, which necessarily involves an arc of $T$. Consequently, performing a smoothing at the first coordinate always corresponds to resolving a crossing formed by two arcs, rather than a crossing between a circle and the arc. This process is illustrated in Figure~\ref{figure:arc_smoothing}.
\begin{figure}[h]
  \centering
   \includegraphics[width=0.15\textwidth]{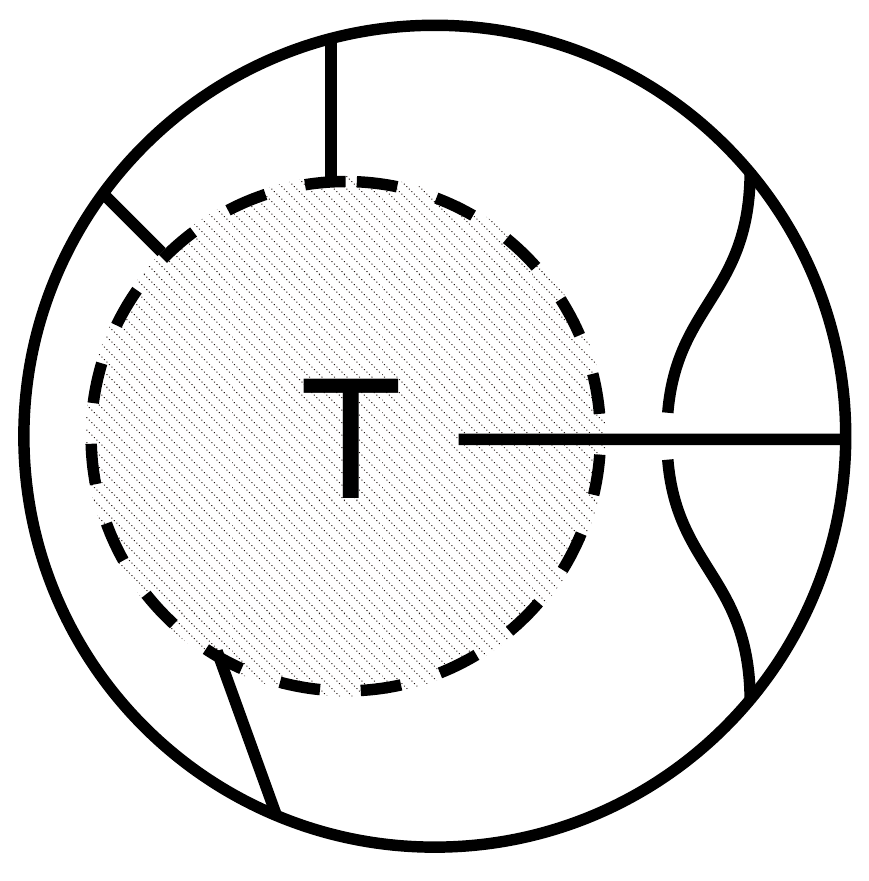}\\
  \caption{Illustration of the tangle $T'$ obtained from $T$ by adding an arc with a single crossing.}\label{figure:arc_smoothing}
\end{figure}

Hence we obtain
\begin{align*}
    & T'_{(0,s)} = T_{s} \sqcup \Omega, \\
    & T'_{(1,s)} = T_{s} \sqcup \Omega,
\end{align*}
where $\Omega$ denotes an arc. Next, we consider the map
\[
   d_{\xi} : T'_{u} \longrightarrow T'_{u'},
\]
where $\xi \in \{0,1,\star\}^{n+1}$ represents an edge of the state cube connecting the states $u$ and $u'$. The nontrivial maps $d_{\xi}$ on $T'_{(0,s)}$ fall into exactly two cases:
\begin{align*}
    & d_{(0,\zeta_1)}: \; T'_{(0,s)} \longrightarrow T'_{(0,s')}, \\
    & d_{(\star,\zeta_2)}: \; T'_{(0,s)} \longrightarrow T'_{(1,s)},
\end{align*}
where $\zeta_1 \in \{0,1,\star\}^{n}$ and $\zeta_2 \in \{0,1\}^{n}$.
The induced map
\[
   \mathcal{G}(d_{(\star,\zeta_2)}): \mathcal{G}(T'_{(0,s)}) \longrightarrow \mathcal{G}(T'_{(1,s)})
\]
is necessarily the zero map. Indeed, smoothing at the first coordinate corresponds to resolving a crossing of two arcs, so the associated morphism takes the form
\[
   \mathcal{G}(d_{(\star,\zeta_2)}): W \otimes W \otimes G_{s} \;\longrightarrow\; W \otimes W \otimes G_{s},
   \quad w \otimes w \otimes x \longmapsto 0,
\]
where $G_{s}$ denotes the TQFT construction corresponding to the remaining part of the state $s$. Therefore, every map $\mathcal{G}(d_{\xi})$ acting on $\mathcal{G}(T'_{(0,s)})$ lands in some $\mathcal{G}(T'_{(0,s')})$.

Let
\[
   A = \bigoplus_{s} \mathcal{G}(T'_{(0,s)}),
   \qquad
   B = \bigoplus_{s} \mathcal{G}(T'_{(1,s)}).
\]
It then follows that
\[
   \mathcal{G}([[T']]) \;=\; A \oplus B.
\]
Recall that the differential on $\mathcal{G}([[T']])$ is given by
\[
   \partial_{T'} = \sum_{\xi} \mathcal{G}(d_{\xi}),
\]
where the sum runs over all edges of the state cube of $T'$. Hence we obtain
\begin{align*}
   & \partial_{T'}(A) \subseteq A, \\
   & \partial_{T'}(B) \subseteq B.
\end{align*}
This shows that both $(A, \partial_{T'}|_{A})$ and $(B, \partial_{T'}|_{B})$ are subcomplexes of $(\mathcal{G}([[T']]), \partial_{T'})$. It follows that
\begin{equation}\label{equation:isomorphism}
  (\mathcal{G}([[T']]), \partial_{T'}) = (A, \partial_{T'}|_{A})\oplus (B, \partial_{T'}|_{B}).
\end{equation}

We now proceed by considering the cases separately.

\medskip
\noindent\textbf{(i) The first coordinate corresponds to a right-handed crossing.}
In this case, we have
\[
   n_{+}(T') = n_{+}(T) + 1,
   \qquad
   n_{-}(T') = n_{-}(T).
\]
By Equation~\eqref{equation:isomorphism}, the Khovanov complex decomposes as
\[
   Kh^{\ast}(T') \;\cong\; Kh^{\ast}(T \sqcup \Omega) \oplus Kh^{\ast-1}(T \sqcup \Omega).
\]
Here, the shift $\ast-1$ arises from the fact that when $\ell(1,s) = k$, one has $\ell(s) = k-1$.
Moreover,
\[
   Kh^{\ast}(T \sqcup \Omega) \;\cong\; Kh^{\ast}(T)\otimes W.
\]
Thus there is an isomorphism
\[
   \rho: H^{k}(T') \;\stackrel{\cong}{\longrightarrow}\; H^{k}(T) \oplus H^{k-1}(T).
\]

Let $[x]\in H^{k}(T')$ and denote its quantum grading by $\Phi([x])$. Then
\[
   \Phi([x]) \;=\; k + n_{+}(T') - n_{-}(T') + \theta(x).
\]
The element $[x]$ corresponds to an element in $H^{\ast}(T)\otimes W$, and under the projection to $H^{\ast}(T)$ via $\rho([x])$, its quantum degree increases by one since $w$ has quantum degree $-1$. Hence
\[
   \theta(\rho([x])) = \theta(x) + 1.
\]

If $[x]\in H^{k}(T)$, then its quantum degree under $\rho$ is
\[
   \Phi_{1}([x])
   = k + n_{+}(T) - n_{-}(T) + (\theta(x)+1)
   = \Phi([x]).
\]
If $[x]\in H^{k-1}(T)$, then in $H^{k-1}(T)\otimes W$ we have
\[
   \Phi_{2}([x])
   = (k-1) + n_{+}(T) - n_{-}(T) + (\theta(x)+1)
   = \Phi([x]) - 1.
\]
Therefore, we have
\[
   H^{k,q}(T') \;\cong\; H^{k,q}(T) \oplus H^{k-1,q-1}(T),
\]
where $q$ denotes the quantum grading.

\medskip
\noindent\textbf{(ii) The first coordinate corresponds to a left-handed crossing.}
In this case, we have
\[
   n_{+}(T') = n_{+}(T),
   \qquad
   n_{-}(T') = n_{-}(T) + 1.
\]
Equation~\eqref{equation:isomorphism} yields
\[
   Kh^{\ast}(T') \;\cong\; Kh^{\ast}(T \sqcup \Omega) \oplus Kh^{\ast-1}(T \sqcup \Omega).
\]
Thus there is an isomorphism
\[
   \rho: H^{k}(T') \;\stackrel{\cong}{\longrightarrow}\; H^{k+1}(T) \oplus H^{k}(T).
\]

Let $[x]\in H^{k}(T')$ with quantum grading $\Phi([x])$.
If $[x]\in H^{k+1}(T)$, then under $\rho$ its quantum degree is
\[
   \Phi_{1}([x])
   = (k+1) + n_{+}(T) - n_{-}(T) + (\theta(x)+1)
   = \Phi([x]) + 3.
\]
If $[x]\in H^{k}(T)$, then its quantum degree is
\[
   \Phi_{2}([x])
   = k + n_{+}(T) - n_{-}(T) + (\theta(x)+1)
   = \Phi([x]) + 2.
\]
Therefore, one has
\[
   H^{k,q}(T') \;\cong\; H^{k+1,q+3}(T) \oplus H^{k,q+2}(T).
\]

Combining the above cases, we complete the proof.
\end{proof}

\subsection{Simple tangles}

\begin{definition}
Let $T$ be a tangle, consisting of arcs and circles. An \emph{arc} of $T$ is called \textbf{pure} if for any point on the arc, each of the two sides of the arc can be connected to the boundary of $T$ by a path that does not intersect any other component of $T$.
\end{definition}

\begin{example}
As illustrated in Figure~\ref{figure:simple_tangle}, in the left tangle none of the arcs is pure. For instance, for arc $a$, one side can be connected to the boundary by a path without crossings,
while on the other side no such crossing-free path exists. In contrast, in the right tangle all arcs are pure; for example, both sides of arc $b$ can be connected to the boundary through paths without crossings.
\begin{figure}[h]
  \centering
  \includegraphics[width= 0.4\textwidth]{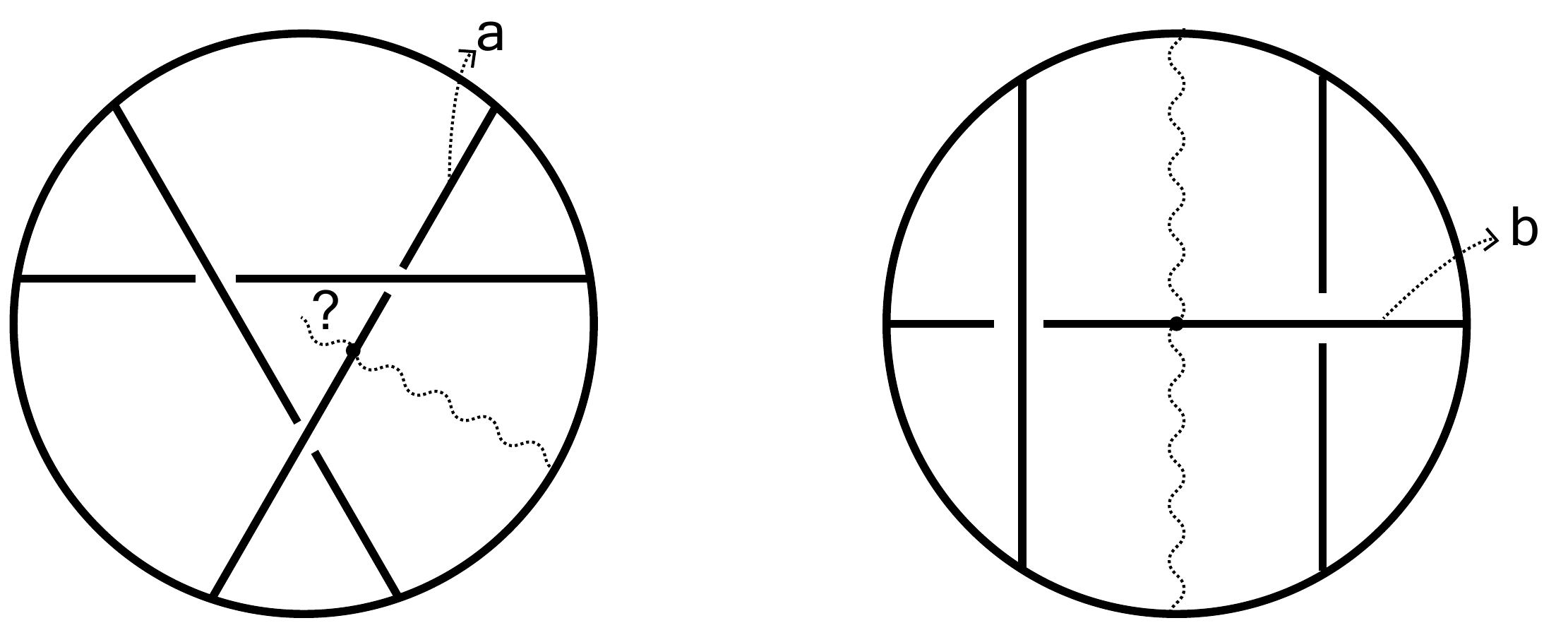}\\
  \caption{The left diagram depicts a tangle without pure arcs. The right diagram illustrates a simple tangle.}\label{figure:pure_arc}
\end{figure}
\end{example}

\begin{definition}
Let $T$ be a tangle. We say that $T$ is a \textbf{simple tangle} if every arc of $T$ is pure.
\end{definition}

Intuitively, a tangle is simple if no collection of arcs forms a closed region within it.

\begin{proposition}\label{prop:simple_tangle}
Let $T$ be a simple tangle. Then we have
\begin{enumerate}[label=(\roman*)]
    \item No arc of $T$ crosses any circle of $T$.
    \item No arc of $T$ has a self-crossing.
    \item Any collection of $N$ arcs in $T$ can produce at most $N-1$ crossings.
\end{enumerate}
\end{proposition}

\begin{proof}
(i) Suppose an arc crosses a circle. Then the portion of the arc lying inside the circle cannot be connected to the boundary of the tangle by a path without crossings, contradicting purity.

(ii) Suppose an arc has a self-crossing. Then the crossing together with a segment of the arc would form a closed region. Any point on that segment cannot be connected to the boundary of the tangle by a path without crossings, again a contradiction.

(iii) Consider the figure formed by the given $N$ arcs and denote by $m$ the number of crossings among them. Whenever two crossings are joined by a finite segment of an arc, we keep this segment. Thus the resulting figure is composed of finitely many arc segments and crossings. Since each crossing involves two arcs, there are $2m$ incident arc ends in total. If two crossings share an incident arc, then these crossings are connected by a finite segment. Hence, among the $2m$ incident arcs, at least $2N-m$ are shared, which implies that there are at least $2N-m$ finite segments. Denote by $l$ the number of such finite segments, so $l \geq 2N-m$.

By Euler's formula for planar graphs we have
\[
   m + f - l = 2,
\]
where $f$ is the number of faces. Since $T$ is simple, these arcs do not form closed regions, hence $f=1$. Therefore
\[
   m + 1 - (2N-m) \leq 2,
\]
which implies $m \leq N-1$. Thus any collection of $N$ arcs in a simple tangle produces at most $N-1$ crossings.
\end{proof}

\begin{corollary}
Let $T$ be a simple tangle. Then $T$ is the disjoint union of a tangle whose components are pure arcs and a link.
\end{corollary}

\begin{proof}
By Proposition~\ref{prop:simple_tangle}\,(i), arcs do not cross circles, hence $T$ decomposes as $A \sqcup C$, where $A$ is the union of all arcs and $C$ the union of all circles. By the definition of a simple tangle, every arc is pure, so $A$ is a tangle of pure arcs. The component $C$ is a link, which completes the proof.
\end{proof}

\begin{proposition}\label{proposition:simple_tangle}
Let $T$ be a connected tangle consisting of $N$ arcs and $N-1$ crossings. Then $T$ is a simple tangle.
\end{proposition}

\begin{proof}
Suppose, for contradiction, that $T$ is not simple. Then there exists a collection of $n$ arcs forming a closed region. These $n$ arcs contribute at least $n$ crossings. Since $T$ is connected, the remaining $N-n$ arcs must also contribute at least $N-n$ crossings. Hence the total number of crossings is at least
\[
n+(N-n) = N,
\]
which contradicts the assumption that $T$ has exactly $N-1$ crossings. Therefore, $T$ must be simple.
\end{proof}

Let $G$ be a group and $\mathbb{K}$ a field. The \emph{group algebra} $\mathbb{K}[G]$ is the $\mathbb{K}$-vector space with basis $\{\,v_g \mid g \in G\,\}$, consisting of all finite linear combinations
\[
\sum_{g \in G} a_g v_g, \qquad a_g \in \mathbb{K},
\]
with multiplication defined by $v_g \cdot v_h = v_{gh}$ and extended bilinearly.

The algebra $\mathbb{K}[G]$ is associative and unital, with the unit given by $v_e$, where $e$ denotes the identity element of $G$.

\begin{example}
If $G = \mathbb{Z}/2\mathbb{Z} = \{0,1\}$ with addition mod $2$, then $\mathbb{K}[G]$ has basis $\{e_0, e_1\}$ and multiplication
\[
e_1 \cdot e_1 = e_0.
\]
Thus, $\mathbb{K}[G] \cong \mathbb{K}[x]/(x^2-1)$, via the identification $x \mapsto e_1$.
\end{example}

\begin{example}
If $G = \mathbb{Z}$, then we have
\[
\mathbb{K}[G] \cong \mathbb{K}[x, x^{-1}].
\]
This is the algebra of Laurent polynomials, with $x$ corresponding to the generator $1 \in \mathbb{Z}$.
\end{example}

\begin{lemma}\label{lemma:tangle_reduction}
Let $T$ be a simple tangle. Then there exists an arc in $T$ that intersects all other arcs in at most one crossing.
\end{lemma}

\begin{proof}
Suppose to the contrary that every arc in $T$ meets other arcs in at least two crossings. If $T$ has $N$ arcs, then the total number of crossings $m$ satisfies
\[
  m \;\geq\; \frac{2N}{2} \;=\; N,
\]
since each crossing involves two arcs. This contradicts Proposition~\ref{prop:simple_tangle}, which asserts that any $N$ arcs produce at most $N-1$ crossings. Hence the lemma follows.
\end{proof}

Consider $(k,q)\in \mathbb{Z}\times \mathbb{Z}$. Define the multiplication
\[
   (k,q)\cdot (k',q') = (k+k',\, q+q').
\]
This operation is commutative, and hence $(\mathbb{Z}\times \mathbb{Z},\cdot)$ is an abelian group.
Now consider the group algebra $\mathbb{K}[\mathbb{Z}\times \mathbb{Z}]$. The multiplication in this algebra is induced by the above group law.

\begin{theorem}\label{theorem:simple_tangle}
Let $T$ be a simple tangle consisting of $N$ arcs, endowed with an orientation. Suppose that $T$ has $n_{+}$ right-handed crossings and $n_{-}$ left-handed crossings. Then the Khovanov homology of $T$ has $2^{n_{+}+n_{-}}$ generators, whose homological and quantum gradings are determined by the expansion of
\[
    (0,-1)^{\,N-n_{+}-n_{-}}\big[(0,0)+(1,1)\big]^{n_{+}}\big[(-1,-3)+(0,-2)\big]^{n_{-}}.
\]
\end{theorem}

\begin{proof}
Since $T$ is simple, Lemma~\ref{lemma:tangle_reduction} ensures that there exists an arc in $T$ that intersects all other arcs in at most one crossing.

If this arc has no crossings with the rest of $T$, then
\[
  H^{k}(T) \;\cong\; W \otimes H^{k}(T_1),
\]
where $T_1$ is the tangle obtained from $T$ by removing this arc. In terms of the quantum grading, we have
\[
  H^{k,q}(T) \;\cong\; H^{k,q+1}(T_1).
\]

If this arc meets the rest of $T$ in exactly one crossing, then by Theorem~\ref{theorem:arc_reduction} we obtain
\begin{itemize}
    \item If the crossing is right-handed, then
    \[
       H^{k,q}(T) \;\cong\; H^{k,q}(T_1) \oplus H^{k-1,q-1}(T_1).
    \]
    \item If the crossing is left-handed, then
    \[
       H^{k,q}(T) \;\cong\; H^{k+1,q+3}(T_1) \oplus H^{k,q+2}(T_1).
    \]
\end{itemize}

Thus, in the right-handed case, each generator $(k,q)$ of $H(T_1)$ gives rise to generators
\[
   (k,q)\big[(0,0)+(1,1)\big],
\]
while in the left-handed case we obtain
\[
   (k,q)\big[(-1,-3)+(0,-2)\big].
\]

Hence, a single arc reduction introduces a multiplicative factor in the generating polynomial, which is
\[
  (0,-1), \quad \big[(0,0)+(1,1)\big], \quad \big[(-1,-3)+(0,-2)\big],
\]
corresponding respectively to removing a free arc, a right-handed crossing, or a left-handed crossing.

Since $T_1$ remains a simple tangle, we may continue this reduction process until only one arc is left. Performing $n_{+}$ right-handed reductions, $n_{-}$ left-handed reductions, and $j$ free arc reductions yields
\[
   (0,-1)(0,-1)^j \big[(0,0)+(1,1)\big]^{n_{+}} \big[(-1,-3)+(0,-2)\big]^{n_{-}},
\]
where $(0,-1)$ corresponds to the single remaining arc.

Noting that each reduction removes one arc, the number of free arc reductions is
\[
  j \;=\; N-n_{+}-n_{-}-1.
\]
Therefore the generators of $H(T)$ are encoded by
\[
    (0,-1)^{\,N-n_{+}-n_{-}} \big[(0,0)+(1,1)\big]^{n_{+}} \big[(-1,-3)+(0,-2)\big]^{n_{-}},
\]
and every term in its expansion corresponds to one generator, with its homological and quantum gradings determined by the exponents.
\end{proof}

\begin{definition}[\cite{przytycki2024lectures}]
Let $H = \bigoplus_{i,j} H^{i,j}$ be a bigraded cohomology vector space, where $H^{i,j}$ denotes the subspace of cohomology classes of bidegree $(i,j)$.
The \textbf{Poincar\'e polynomial} of $H$ is defined by
\[
\mathcal{P}_H(x,y) = \sum_{i,j} (\dim H^{i,j}) \, x^i y^j.
\]
Here, $\dim H^{i,j}$ counts the number of generators in bidegree $(i,j)$.
\end{definition}
In the case of Khovanov homology or other bigraded cohomology theories, $\mathcal{P}_H(x,y)$ encodes the distribution of generators over homological and quantum gradings. In particular, by setting $x=-1$, the Poincar\'e polynomial of the homology of tangles specializes to the Jones polynomial.

\begin{theorem}\label{theorem:poincare}
Let $T$ be a simple tangle consisting of $N$ arcs, endowed with an orientation. Suppose that $T$ has $n_{+}$ right-handed crossings and $n_{-}$ left-handed crossings. Then the Poincar\'e polynomial of $T$ is given by
\[
  \mathcal{P}_{T}(x,y)=y^{-N+n_{+}+n_{-}}(1+xy)^{n_{+}}\,(x^{-1}y^{-3}+y^{-2})^{n_{-}}.
\]
\end{theorem}

In Theorem \ref{theorem:poincare}, each monomial $x^k y^q$ in its expansion corresponds to a generator of the tangle's Khovanov homology, with homological degree $k$ and quantum grading $q$.

\subsection{Examples}

\begin{example}
In this example, we choose an orientation so that the tangle has $n_{+}$ right-handed crossings and no left-handed crossings. Then the tangle contains $n_{+}+1$ arcs, as shown in Figure~\ref{figure:simple_tangle}.
\begin{figure}[h]
  \centering
  \includegraphics[width=0.3\textwidth]{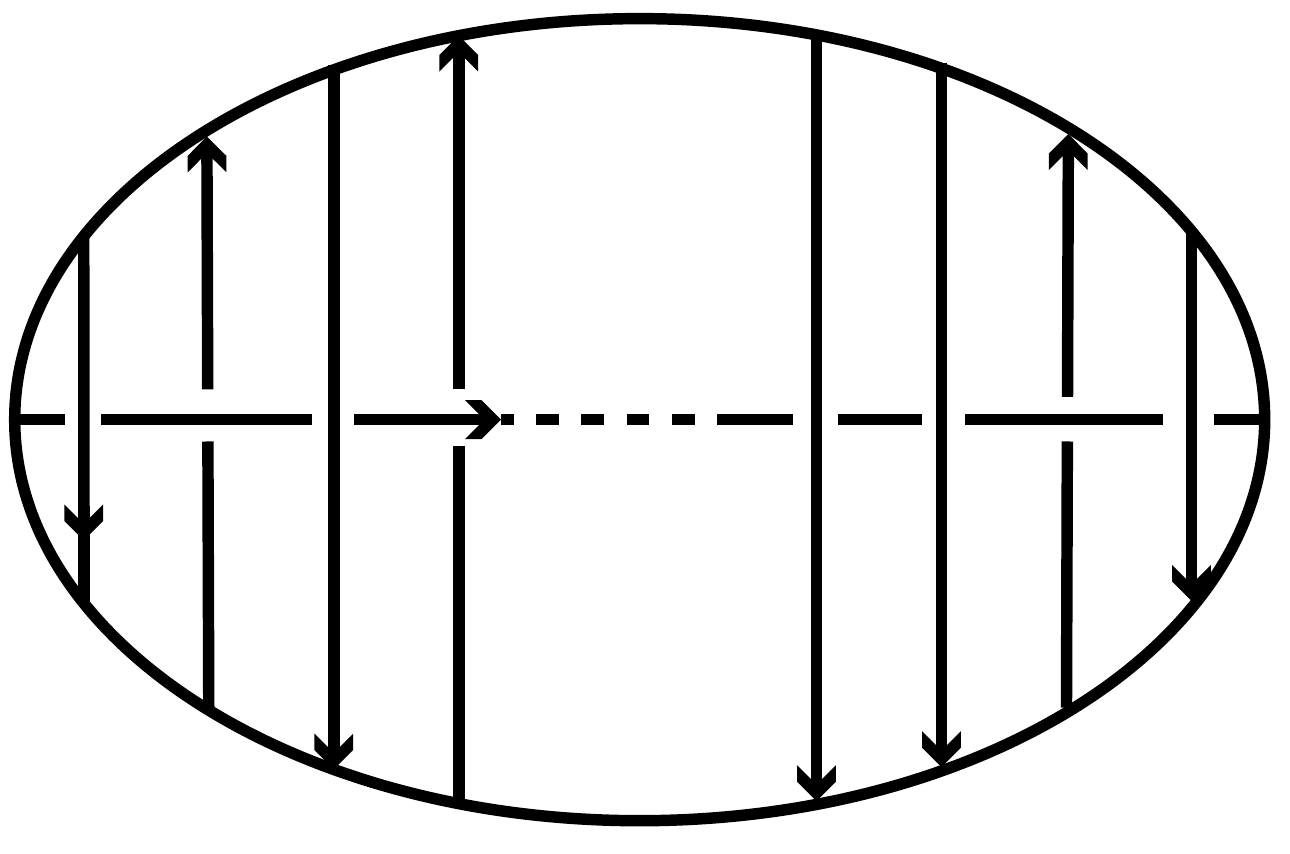}\\
  \caption{Illustration of a simple tangle in which parallel arcs intersect a given arc.}
  \label{figure:simple_tangle}
\end{figure}
By Theorem~\ref{theorem:simple_tangle}, the Khovanov homology of this tangle has $2^{\,n_{+}}$ generators, corresponding to the expansion of
\[
  (0,-1) \big[(0,0)+(1,1)\big]^{n_{+}}.
\]

Alternatively, this can be interpreted using the Poincar\'e polynomial
\[
  y^{-1}(1+xy)^{\,n_{+}}.
\]
Each monomial $x^k y^q$ in its expansion corresponds to a generator of the Khovanov homology with homological degree $k$ and quantum grading $q$. Note that
\[
   (1+xy)^{\,n_{+}} = \sum_{i=0}^{n_{+}} \binom{n_{+}}{i} x^{i}y^{i}.
\]
This implies that the Khovanov homology of this tangle has $\binom{n_{+}}{k}$ generators with homological degree $k$ and quantum grading $k-1$, for $k = 0,1,\dots,n_{+}$.
\end{example}

\begin{example}\label{example:4arcs}
Consider the tangle $T$ illustrated in Figure~\ref{figure:simple_tangle2}.

\begin{figure}[h]
  \centering
  \includegraphics[width=0.15\textwidth]{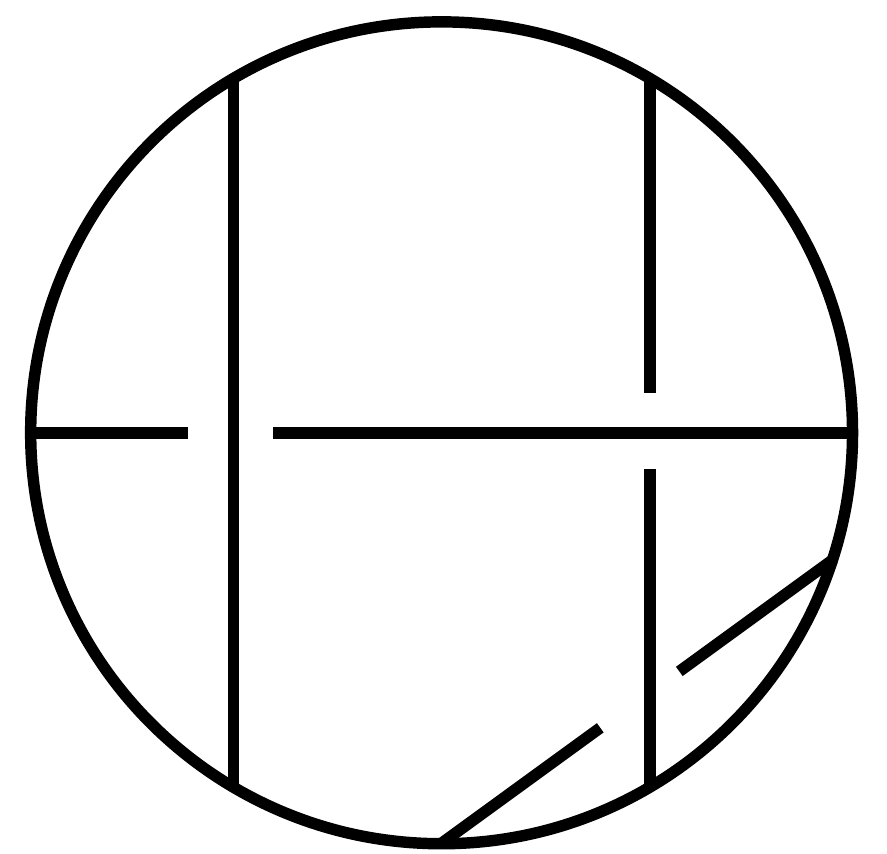}\\
  \caption{Illustration of a simple tangle with four arcs.}
  \label{figure:simple_tangle2}
\end{figure}

This tangle admits $2^4 = 16$ possible orientations, although some of them may be equivalent. By Theorem~\ref{theorem:simple_tangle}, the Khovanov homology of $T$ is determined by the number of right-handed and left-handed crossings. Consequently, there are four distinct cases for the Khovanov homology corresponding to different combinations of crossings. We compute them below.

\begin{enumerate}[label=(\roman*)]
    \item When $n_{+}=0$ and $n_{-}=3$, the distribution of generators is
    \[
      (-3,-10) + 3(-2,-9) + 3(-1,-8) + (0,-7).
    \]
    The corresponding Poincar\'e polynomial is
    \[
      \mathcal{P}_{T}(x,y) = x^{-3}y^{-10} + 3 x^{-2}y^{-9} + 3 x^{-1}y^{-8} + y^{-7}.
    \]
    This polynomial encodes the distribution of generators in homological degree and quantum grading.

    \item When $n_{+}=1$ and $n_{-}=2$, the distribution of generators is
    \[
      (-2,-7) + 3(-1,-6) + 3(0,-5) + (1,-4).
    \]
    The corresponding Poincar\'e polynomial is
    \[
      \mathcal{P}_{T}(x,y) = x^{-2}y^{-7} + 3 x^{-1}y^{-6} + 3 y^{-5} + x y^{-4}.
    \]

    \item When $n_{+}=2$ and $n_{-}=1$, the distribution of generators is
    \[
      (-1,-4) + 3(0,-3) + 3(1,-2) + (2,-1).
    \]
    The corresponding Poincar\'e polynomial is
    \[
      \mathcal{P}_{T}(x,y) = x^{-1}y^{-4} + 3 y^{-3} + 3 x y^{-2} + x^2 y^{-1}.
    \]

    \item When $n_{+}=3$ and $n_{-}=0$, the distribution of generators is
    \[
      (0,-1) + 3(1,0) + 3(2,1) + (3,2).
    \]
    The corresponding Poincar\'e polynomial is
    \[
      \mathcal{P}_{T}(x,y) = y^{-1} + 3 x + 3 x^2 y + x^3 y^2.
    \]
\end{enumerate}
\end{example}

\section{Classification tables for low-crossing tangles}\label{section:classification}

In this section, we study the classification and computation of tangles with at most three crossings. Here, only consider the connected tangles, since the Khovanov homology of a tangle is the tensor product of the Khovanov homologies of its connected components \cite{shen2025algorithm}. Moreover, we consider tangles up to Reidemeister move equivalence, so that equivalent tangles under such moves are identified. In addition, when depicting tangles, the illustrations may not include their mirror-symmetric counterparts. However, in our computations these cases are automatically taken into account by considering the signs of the crossings.

\subsection{Tangles with at most two crossings}

First, a tangle with zero crossings has only two possibilities: either a single arc or a single circle.

Second, a tangle with one crossing has, up to Reidemeister moves, only one type: a crossing formed by two arcs. Intuitively, a single arc could form a self-crossing, but such a crossing can be removed by an $R1$ move, reducing it to a trivial arc.
\begin{figure}[h]
  \centering
  \includegraphics[width=0.6\textwidth]{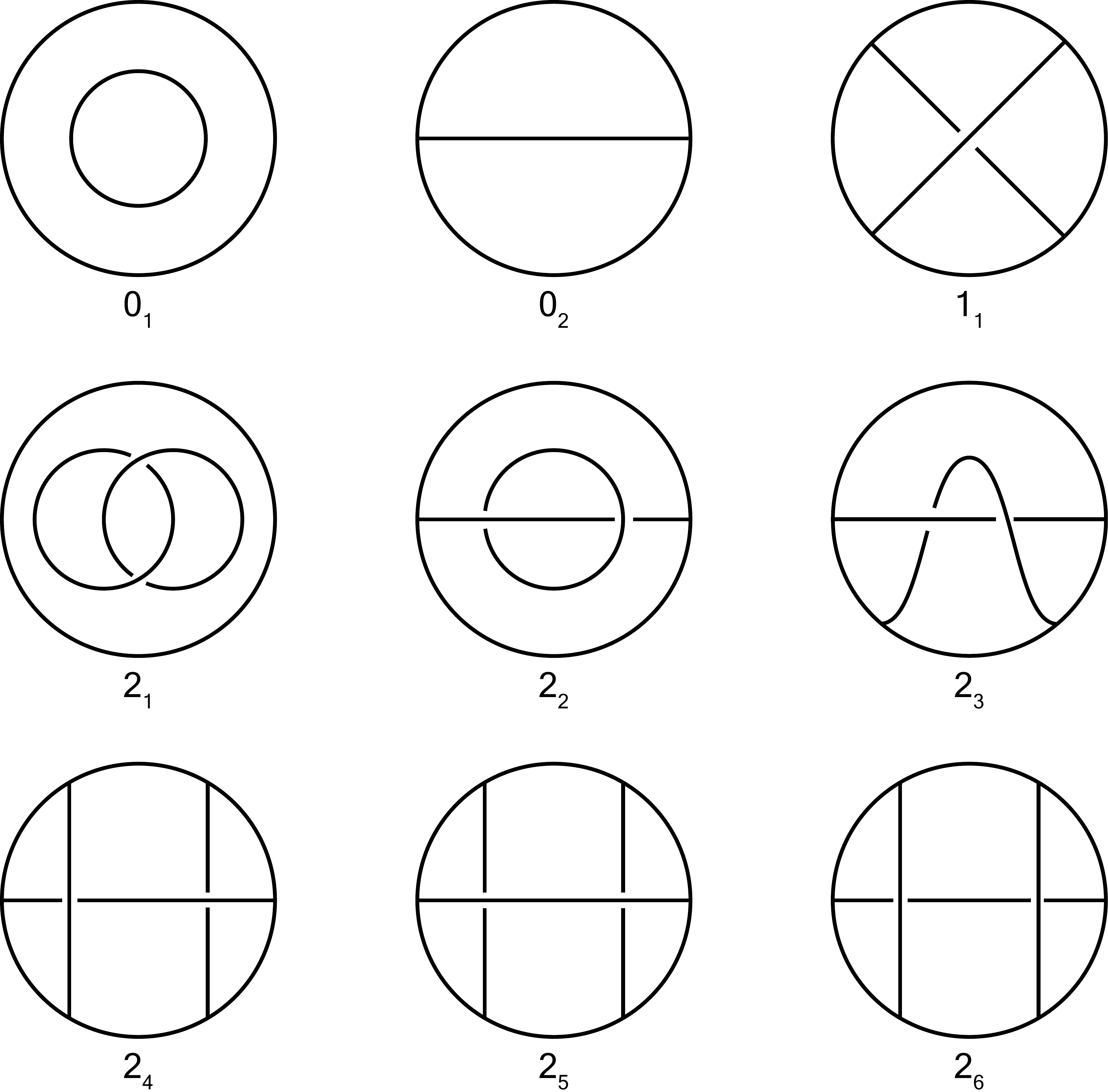}\\
  \caption{Illustration of tangles with 0,1, and 2 crossings.}
  \label{figure:tangle_table_012}
\end{figure}

The case of two crossings is more involved. If the crossings are formed by two circles, the tangle is a Hopf link. If a circle and an arc form a crossing, the tangle corresponds to type $2_2$ in Figure~\ref{figure:tangle_table_012}. If a single arc forms a tangle with two self-crossings, then this tangle can be simplified to a trivial arc by two $R1$ moves. If the tangle is formed by two arcs, the only possibility is type $2_3$ in Figure~\ref{figure:tangle_table_012}; indeed, if the two arcs intersect in only one crossing while one of them forms a self-crossing, the self-crossing can be eliminated by an $R1$ move. Finally, consider the case of three arcs: here, exactly one arc intersects the other two to form crossings, as illustrated by types $2_4$, $2_5$, and $2_6$ in Figure~\ref{figure:tangle_table_012}. These three tangles are not related by Reidemeister moves, and hence they are considered distinct tangles.

The above discussion covers tangles with at most two crossings, without considering orientations. In fact, once orientations are taken into account, tangles can be distinguished into more types.
For knots, the Khovanov homology is independent of orientation; however, for tangles, it is highly sensitive to orientation. The effect of different orientations on the Khovanov homology of a tangle is mainly reflected in the numbers of right-handed and left-handed crossings: if two orientations have the same $(n_{+}, n_{-})$, then their Khovanov homologies coincide \cite[Theorem 4.2]{shen2025algorithm}. Therefore, in Table~\ref{table:knot_types}, we compute the Poincar\'{e} polynomials of each tangle type for the different signs of crossings. Here we also compute the mirror type of a tangle, namely the tangle type obtained by taking the mirror reflection, where $n'_{i}$ denotes the mirror type corresponding to the type of $n_{i}$.

\begin{table}
  \centering
\begin{tabular}{c|c|c}
    \hline
    Tangle type & Sign of crossings & Poincar\'{e} polynomial \\
    \hline    \hline
    $0_0$ & $\emptyset$ & $y+y^{-1}$ \\
    \hline
    $0_1$ & $\emptyset$ & $y^{-1}$ \\
    \hline
    $1_1$ & \makecell{$\{+\}$ \\ $\{-\}$} & \makecell{$x+y^{-1}$  \\ $y^{-3}+x^{-1}y^{-4}$  } \\
    \hline
    $2_1$ & \makecell{$\{+,+\}$ \\ $\{-,-\}$} & \makecell{ $1+y^{2}+x^{2}y^{4}+x^{2}y^{6}$\\  $1+y^{-2}+x^{-2}y^{-4}+x^{-2}y^{-6}$} \\
    \hline
    $2'_1$ & \makecell{$\{+,+\}$ \\ $\{-,-\}$} & \makecell{ $1+y^{2}+x^{2}y^{4}+x^{2}y^{6}$\\  $1+y^{-2}+x^{-2}y^{-4}+x^{-2}y^{-6}$} \\
    \hline
    $2_2$ & \makecell{$\{+,+\}$ \\ $\{-,-\}$} & \makecell{$1+x^{2}y^{4}$ \\ $y^{-2}+x^{-2}y^{-6}$} \\
    \hline
    $2_3$ & \makecell{$\{+,+\}$ \\ $\{-,-\}$} & \makecell{ $1+xy+x^{2}y^{3}$ \\ $y^{-3}+x^{-1}y^{-5}+x^{-2}y^{-6}$} \\
    \hline
    $2'_3$ & \makecell{$\{+,+\}$ \\ $\{-,-\}$} & \makecell{$x^2y^2+xy+y^{-1}$ \\ $x^{-2}y^{-7}+x^{-1}y^{-5}+y^{-4}$} \\
    \hline
    $2_4$, $2_5$, $2_6$ & \makecell{$\{+,+\}$ \\ $\{+,-\}$ \\ $\{-,-\}$} & \makecell{ $x^{2}y+2x+y^{-1}$ \\ $xy^{-2}+2y^{3}+x^{-1}y^{-4}$ \\ $y^{-5}+2x^{-1}y^{-6}+x^{-2}y^{-7}$  } \\
    \hline
\end{tabular}
  \caption{Poincar\'{e} polynomials of tangles with at most two crossings.}\label{table:knot_types}
\end{table}

\subsection{Tangles with three crossings}

Now we turn to tangles with three crossings. We begin with the case of a single component. If the component is closed, we obtain the trefoil knot. Indeed, there are two kind of trefoil with different sign of crossings. If the component is open, then the corresponding tangle is the $3_2$ type in Figure~\ref{figure:tangle_table_3}.

\begin{figure}[h]
  \centering
  \includegraphics[width=0.6\textwidth]{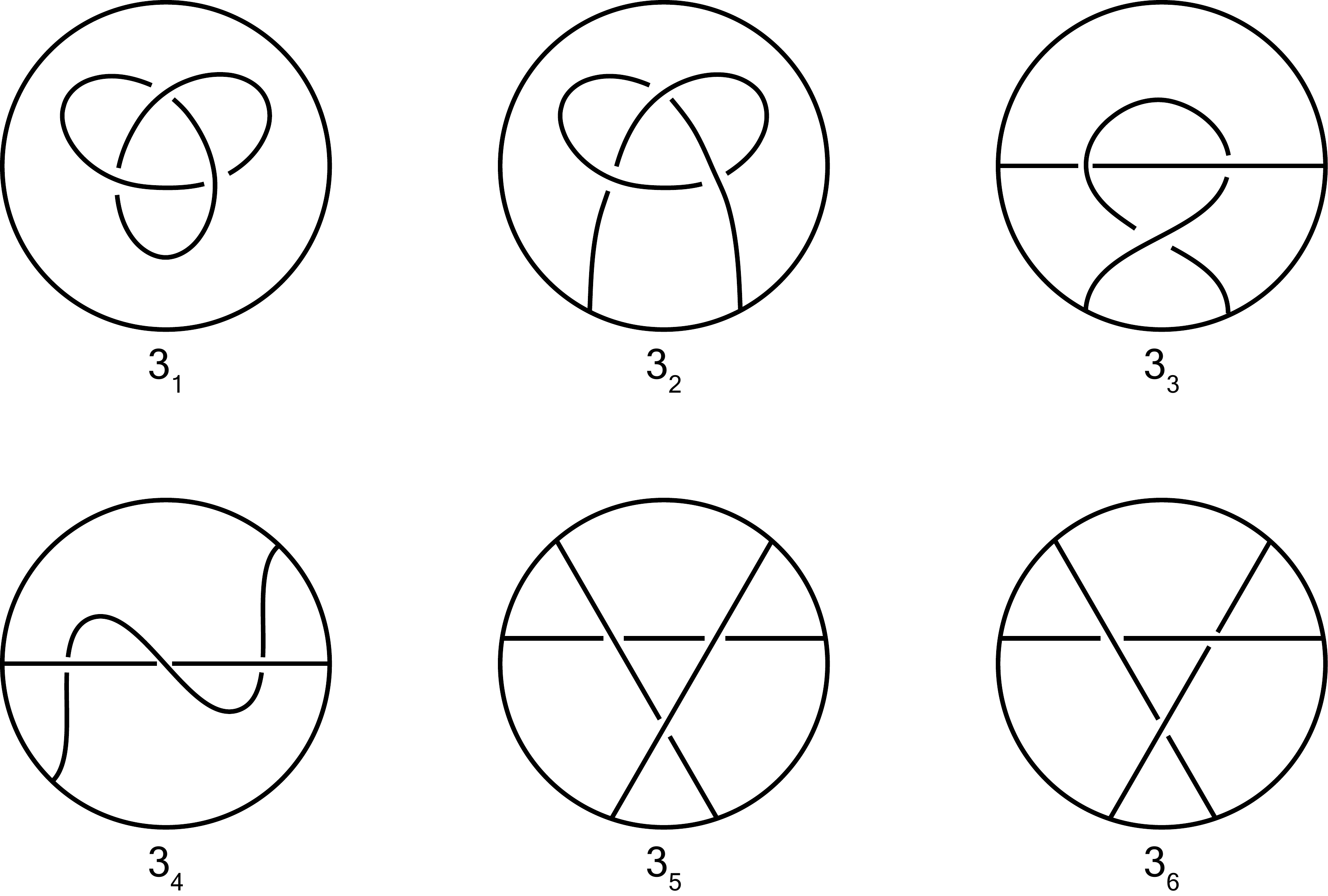}
  \caption{Tangles with three crossings.}
  \label{figure:tangle_table_3}
\end{figure}

Next, consider the case of two components. If one of them is a closed curve, then the other must contribute an even number of crossings. In this situation, the closed component contributes one self-crossing, which can be removed by an $R1$ move. Thus the configuration reduces to the case of two crossings. Therefore both components must be arcs. If one arc has a self-crossing, the only possibility is the $3_3$ type in Figure~\ref{figure:tangle_table_3}. Note that if the horizontal arc lies entirely above or below the other arc, an $R2$ move reduces the diagram. Hence the horizontal arc must pass through the other arc. If it passes in the alternative way shown in Figure~\ref{figure:counter_tangle}, then one crossing can be removed by a sequence of $R3$ and $R1$ moves.

\begin{figure}[h]
  \centering
  \includegraphics[width=0.4\textwidth]{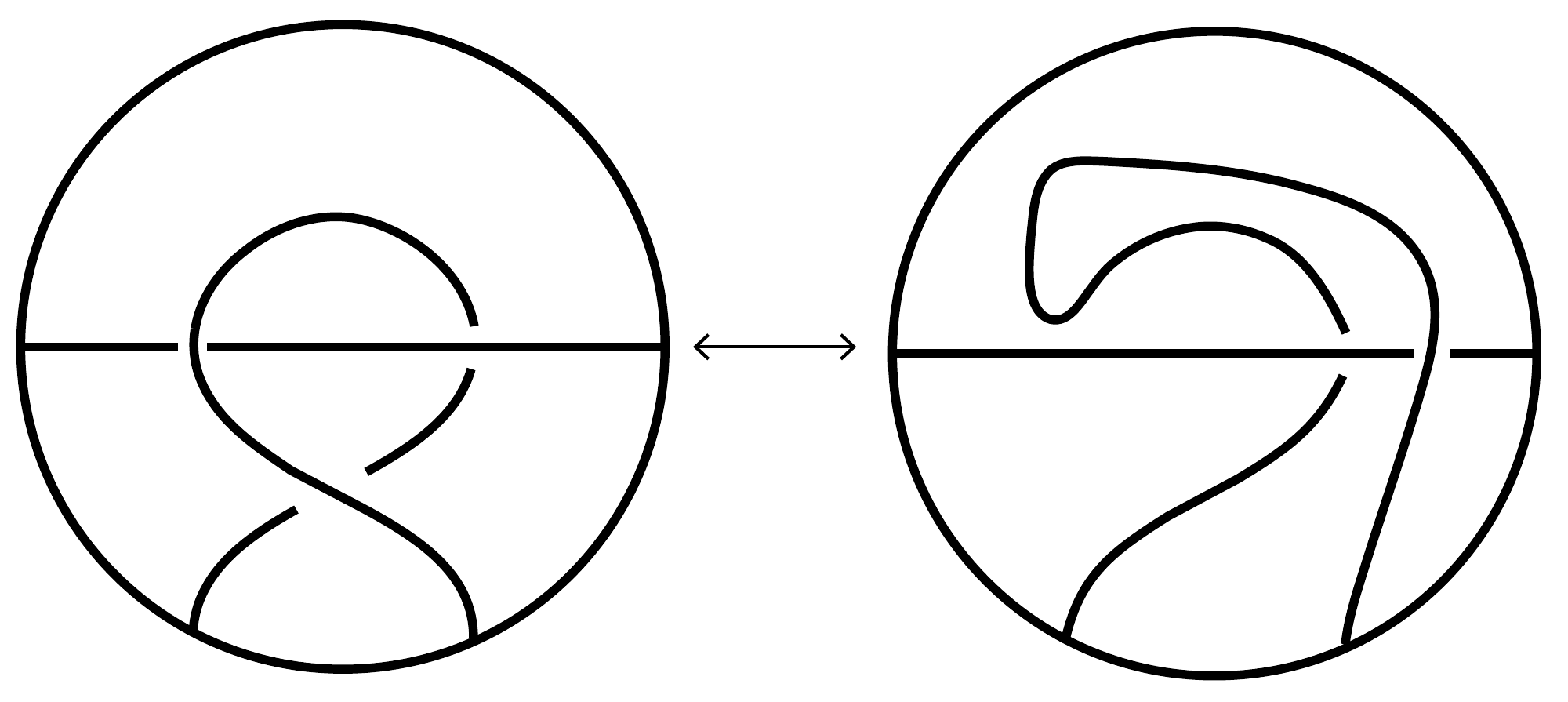}
  \caption{Illustration of the tangle reduction.}
  \label{figure:counter_tangle}
\end{figure}

If one arc intersects the other twice, the diagram corresponds to the $3_4$ type in Figure~\ref{figure:tangle_table_3}; all other possibilities reduce by Reidemeister moves. Finally, if all three components are arcs, then up to Reidemeister moves only the two configurations $3_5$ and $3_6$ in Figure~\ref{figure:tangle_table_3} occur.

Consider the case of a tangle with four arcs and three crossings. Note that . By Proposition~\ref{proposition:simple_tangle}, the tangle must be simple. We can then apply Theorem~\ref{theorem:poincare} to compute its Poincar\'e polynomial. The result can be obtained from the computation in Example~\ref{example:4arcs}.

\begin{table}
  \centering
\begin{tabular}{c|c|c}
    \hline
    Tangle type & Sign of crossings & Poincar\'{e} polynomial \\
    \hline    \hline
    $3_1$ &  \makecell{$\{-,-,-\}$} & $x^{-3}y^{-9}+x^{-2}y^{-5}+y^{-3}+y^{-1}$\\
    \hline
    $3'_1$ &  \makecell{$\{+,+,+\}$} & $x^3y^9+x^2y^5+y^3+y$\\
    \hline
    $3_2$ & \makecell{$\{-,-,-\}$ }  &  $x^{-3}y^{-9}+x^{-2}y^{-7}+y^{-3}$\\
    \hline
    $3'_2$ & \makecell{$\{+,+,+\}$}  &  $x^3y^7+x^2y^5+y$\\
    \hline
    $3_3$ & \makecell{$\{-,-,-\}$ \\ $\{+,+,-\}$ \\} &\makecell{$x^{-3}y^{-9}+x^{-2}y^{-8}+x^{-2}y^{-7}+x^{-1}y^{-6}+y^{-4}$ \\ $x^2y^2+x+y^{-1}+y^{-2}+x^{-1}y^{-3}$}   \\
    \hline
    $3'_3$ & \makecell{$\{+,+,+\}$ \\ $\{-,-,+\}$ \\} & \makecell{$x^3y^5+x^2y^4+x^2y^3+xy^2+1$\\$x^{-2}y^{-6}+x^{-1}y^{-4}+y^{-3}+y^{-2}+xy^{-1}$}  \\
    \hline
    $3_4$ & \makecell{$\{+,+,+\}$ \\ $\{-,-,-\}$} & \makecell{$x^3y^6+x^2y^4+xy^2+y$\\$x^{-3}y^{-8}+x^{-2}y^{-7}+x^{-1}y^{-5}+y^{-3}$} \\
    \hline
    $3'_4$ & \makecell{$\{-,-,-\}$ \\ $\{+,+,+\}$} &\makecell{$x^{-3}y^{-10}+x^{-2}y^{-8}+x^{-1}y^{-6}+y^{-5}$\\$x^3y^4+x^2y^3+xy+y^{-1}$}  \\
    \hline
    $3_5$ & \makecell{$\{+,+,+\}$ \\ $\{+,-,-\}$} & \makecell{$2x^2y^2+2xy+1$\\ $x^{-2}y^{-6}+2x^{-1}y^{-5}+2y^{-4}$} \\
    \hline
      $3'_5$ & \makecell{$\{-,-,-\}$ \\ $\{-,+,+\}$} & \makecell{$2x^{-2}y^{-8}+2x^{-1}y^{-7}+y^{-6}$\\$x^2+2xy^{-1}+2y^{-2}$ } \\
    \hline

    $3_6$ & \makecell{$\{-,-,-\}$ \\ $\{+,+,-\}$} & \makecell{$x^{-3}y^{-9}+3x^{-2}y^{-8}+2x^{-1}y^{-7}+y^{-5}$\\ $x^2y+2xy^{-1}+3y^{-2}+x^{-1}y^{-3}$ } \\
   \hline
     $3'_6$ & \makecell{$\{+,+,+\}$ \\ $\{-,-,+\}$} & \makecell{$x^3y^3+3x^2y^2+2xy+y^{-1}$\\$xy^{-3}+3y^{-4}+2x^{-1}y^{-5}+x^{-2}y^{-7}$ } \\
     \hline
     Tangle with 4 arcs & \makecell{$\{+,+,+\}$ \\ $\{+,+,-\}$\\ $\{+,-,-\}$\\ $\{-,-,-\}$} & \makecell{$y^{-1} + 3 x + 3 x^2 y + x^3 y^2$\\$x^{-1}y^{-4} + 3 y^{-3} + 3 x y^{-2} + x^2 y^{-1}$\\ $x^{-2}y^{-7} + 3 x^{-1}y^{-6} + 3 y^{-5} + x y^{-4}$\\$x^{-3}y^{-10} + 3 x^{-2}y^{-9} + 3 x^{-1}y^{-8} + y^{-7}$} \\
    \hline
\end{tabular}
  \caption{Poincar\'{e} polynomials of tangles with three crossings.}\label{table:knot_types3}
\end{table}

\section*{Acknowledgments}

This work was supported in part by  Michigan State University Research Foundation and  Bristol-Myers Squibb  65109. Li was supported by an NITMB fellowship supported by grants from the NSF (DMS-2235451) and Simons Foundation (MP-TMPS-00005320). Jian was also supported by the Natural Science Foundation of China (NSFC Grant No. 12401080) and the start-up research fund from Chongqing University of Technology.

\bibliographystyle{plain}  % unsrtnat
\bibliography{Reference}

\end{document}